\documentclass[12pt]{article}
\usepackage{graphicx}
\usepackage{amsmath}
\usepackage{amssymb}
\usepackage{amsthm}
\usepackage{epsfig}
\usepackage{mathrsfs}
\usepackage{multirow}
\usepackage[usenames,dvipsnames]{color}
\usepackage{setspace}
\usepackage{algorithm, algorithmic} 
\usepackage{multicol}
\usepackage{subcaption}
\usepackage[margin=2.54cm]{geometry}
\DeclareMathOperator{\Arg}{Arg}
\DeclareMathOperator{\st}{s.t.}
\DeclareMathOperator{\diag}{diag}

\DeclareMathOperator{\homg}{hom}
\DeclareMathOperator{\SDP}{sdp}

\DeclareMathOperator{\RLT}{RLT}
\DeclareMathOperator{\SOCRLT}{SOCRLT}
\usepackage{colortbl}
\newtheorem{theorem}{Theorem}

\newtheorem{lemma}{Lemma}
\newtheorem{proposition}{Proposition}
\newtheorem{assumption}{Assumption \!\!}

\newtheorem{remark}{Remark}
\newtheorem{example}{Example}

\newcommand{\U}{{\mathcal U}}
\newcommand{\V}{{\mathcal V}}

\newcommand{\K}{{\mathcal K}}

\newcommand{\CP}{\mathcal{CP}}
\newcommand{\COP}{\mathcal{COP}}
\newcommand{\SYM}{{\mathcal S}}
\newcommand{\nin}{\noindent}
\def\RR{ {\Bbb{R}}}

\title{Robust Sensitivity Analysis of the \\ Optimal Value of Linear Programming}

\author{
Guanglin Xu\thanks{Department of Management Sciences, University of Iowa,
Iowa City, IA, 52242-1994, USA. Email: {\tt guanglin-xu@uiowa.edu}.}%
\ \ \ \ \ \ \
Samuel Burer\thanks{Department of Management Sciences, University of Iowa,
Iowa City, IA, 52242-1994, USA. Email: {\tt samuel-burer@uiowa.edu}.}%
}

\date{September 14, 2015 \\ Revised: November 4, 2015}

\begin{document}

\maketitle

\begin{abstract}

\noindent We propose a framework for sensitivity analysis of linear
programs (LPs) in minimization form, allowing for simultaneous
perturbations in the objective coefficients and right-hand sides, where
the perturbations are modeled in a compact, convex uncertainty set. This
framework unifies and extends multiple approaches for LP sensitivity
analysis in the literature and has close ties to worst-case linear
optimization and two-stage adaptive optimization. We define the minimum
(best-case) and maximum (worst-case) LP optimal values, $p^-$ and
$p^+$, over the uncertainty set, and we discuss issues of finiteness,
attainability, and computational complexity. While $p^-$ and $p^+$ are
difficult to compute in general, we prove that they equal the optimal
values of two separate, but related, copositive programs. We then
develop tight, tractable conic relaxations to provide lower and upper
bounds on $p^-$ and $p^+$, respectively. We also develop techniques
to assess the quality of the bounds, and we validate our approach
computationally on several examples from---and inspired by---the
literature. We find that the bounds on $p^-$ and $p^+$ are very strong
in practice and, in particular, are at least as strong as known results
for specific cases from the literature.

\mbox{}


\end{abstract}

%\begin{keyword}

%\end{keyword}

\begin{onehalfspace}

\section{Introduction} \label{sec:intro}

The standard-form linear program (LP) is

\begin{equation} 
\begin{array}{ll} \label{Eq:StdPrimal}
\min & \hat c^Tx  \\
\st & \hat Ax = \hat b  \\
& x \ge 0
\end{array}
\end{equation}

\nin where $x \in \RR^n$ is the variable and $(\hat A, \hat b,
\hat c) \in \RR^{m\times n} \times \RR^m \times \RR^n$ are the
problem parameters. In practice, $(\hat A, \hat b, \hat c)$ may not
be known exactly or may be predicted to change within a certain
region. In such cases, {\em sensitivity analysis} (SA) examines how
perturbations in the parameters affect the optimal value and solution
of (\ref{Eq:StdPrimal}). Ordinary SA considers the change of a single
element in $(\hat A, \hat b, \hat c)$ and examines the corresponding
effects on the optimal basis and tableau; see \cite{Dantzig}. SA also
extends to the addition of a new variable or constraint, although we do
not consider such changes in this paper.

Beyond ordinary SA, more sophisticated approaches that allow
simultaneous changes in the coefficients $\hat c$ or right-hand sides
$\hat b$ have been proposed by numerous researchers. Bradley et
al.~\cite{Bradley} discuss the {\em 100-percent rule} that requires
specification of directions of increase or decrease from each $\hat
c_j$ and then guarantees that the same basis remains optimal as long
as the sum of fractions, corresponding to the percent of maximum
change in each direction derived from ordinary SA, is less than or
equal to 1. Wendell \cite{Wendell1, Wendell2, Wendell3} develops the
{\em tolerance approach} to find the so-called {\em maximum tolerance
percentage} by which the objective coefficients can be simultaneously
and independently perturbed within {\em a priori} bounds. The tolerance
approach also handles perturbations in one row or column of the
matrix coefficients \cite{Ravi1} or even more general perturbations
in all elements of the matrix coefficients under certain assumptions
\cite{Ravi2}. Freund \cite{Freund.1985} investigates the sensitivity of
an LP to simultaneous changes in matrix coefficients. In particular,
he considers a linear program whose coefficient matrix depends
linearly on a scalar parameter $\theta$ and studies the effect of
small perturbations on the the optimal objective value and solution;
see also \cite{Kim.1971, Klatte.1979, Orchard-Hayes.1968}. Readers
are referred to \cite{Wendell4} for a survey of approaches for SA of
problem (\ref{Eq:StdPrimal}).

An area closely related to SA is {\em interval linear programming}
(ILP), which can be viewed as {\em multi-parametric linear programming}
with independent interval domains for the parameters \cite{Gal2, Gal1,
McKeown}. Steuer \cite{Steuer} presents three algorithms for solving
LPs in which the objective coefficients are specified by intervals, and
Gabrel et al.~\cite{Gabrel} study LPs in which the right-hand sides
vary within intervals and discuss the maximum and minimum optimal
values. Mraz \cite{Mraz} considers a general situation in which the
matrix coefficients and right-hand sides change within intervals and
calculates upper and lower bounds for the associated optimal values. A
comprehensive survey of ILP has been given by Hladik \cite{Hladik}.

To the best of our knowledge, in the context of LP, no authors have
considered simultaneous LP parameter changes in a general way, i.e.,
perturbations in the objective coefficients $\hat c$, right-hand sides
$\hat b$, and constraint coefficients $\hat A$ within a general region
(not just intervals). The obstacle for doing so is clear: general
perturbations lead to nonconvex quadratic programs (QPs), which are 
NP-hard to solve (as discussed below).

In this paper, we extend---and in many cases unify---the SA literature
by employing modern tools for nonconvex QPs. Specifically, we
investigate SA for LPs in which $(\hat b, \hat c)$ may change within a
general compact, convex set $\U$, called the {\em uncertainty set\/}.
Our goal is to calculate---or bound---the corresponding minimum
(best-case) and maximum (worst-case) optimal values. Since these values
involve the solution of nonconvex QPs, we use standard techniques from {\em
copositive optimization\/} to reformulate these problems into convex
{\em copositive programs} (COPs), which provide a theoretical grounding
upon which to develop tight, tractable convex relaxations. We suggest
the use of {\em semidefinite programming} (SDP) relaxations, which also
incorporate valid conic inequalities that exploit the structure of the
uncertainty set. We refer the reader to \cite{Dur.2010} for a survey on
copositive optimization and its connections to semidefinite programming.
Relevant definitions and concepts will also be given in this paper; see
Section \ref{ssec:notation}. 

Our approach is related to the recent work on {\em worst-case linear
optimization\/} introduced by Peng and Zhu \cite{Peng.Zhu.2015} in
which: (i) only $\hat b$ is allowed to change within an ellipsoidal
region; and (ii) only the worst-case LP value is considered. (In
fact, one can see easily that, in the setup of \cite{Peng.Zhu.2015}
based on (i), the best-case LP value can be computed in polynomial
time via second-order-cone programming, making it less interesting to
study in their setup.) The authors argue that the worst-case value is
NP-hard to compute and use a specialized nonlinear semidefinite program
(SDP) to bound it from above. They also develop feasible solutions to
bound the worst-case value from below and show through a series of
empirical examples that the resulting gaps are usually quite small.
Furthermore, they also demonstrate that their SDP-based relaxation
is better than the so-called {\em affine-rule approximation} (see
\cite{Ben-Tal.Goryashko.Guslitzer.Nemirovski.2004}) and the Lasserre
{\em linear matrix inequality} relaxation (see \cite{Lasserre.2001,
Henrion.Lasserre.Loefberg.2009}).

Our approach is more general than \cite{Peng.Zhu.2015} because we
allow both $\hat b$ and $\hat c$ to change, we consider more general
uncertainty sets, and we study both the worst- and best-case values. In
addition, instead of developing a specialized SDP approach, we make use
of the machinery of copositive programming, which provides a theoretical
grounding for the construction of tight, tractable conic relaxations
using existing techniques. Nevertheless, we have been inspired by their
approach in several ways. For example, their proof of NP-hardness also
shows that our problem is NP-hard; we will borrow their idea of using
primal solutions to estimate the quality of the relaxation bounds; and
we test some of the same examples.

We mention two additional connections of our approach with the
literature. In \cite{Bertsimas.Goyal.2012}, Bertsimas and Goyal consider
a two-stage adaptive linear optimization problem under right-hand
side uncertainty with a min-max objective. A simplified version of
this problem, in which the first-stage variables are non-existent,
reduces to worst-case linear optimization; see the introduction of
\cite{Bertsimas.Goyal.2012}. In fact, Bertsimas and Goyal use this
fact to prove that their problem is NP-hard via the so-called max-min
fractional set cover problem, which is a specific worst-case linear
optimization problem studied by Feige et al.~\cite{Feige.et.al.2007}.
Our work is also related to the study of {\em adjustable robust
optimization\/} \cite{Ben-Tal.Goryashko.Guslitzer.Nemirovski.2004,
Takeda.Taguchi.Tutuncu.2008}, which allows for two sets of
decisions---one that must be made before the uncertain data is realized,
and one after. In fact, our problem can viewed as a simplified case of
adjustable robust optimization having no first-stage decisions. On the
other hand, our paper is distinguished by its application to sensitivity
analysis and its use of copositive and semidefinite optimization.

We organize the paper as follows. In Section \ref{sec:rsa}, we extend
many of the existing approaches for SA by considering simultaneous,
general changes in $(\hat b, \hat c)$ and the corresponding effect on
the LP optimal value. Precisely, we model general perturbations of
$(\hat b, \hat c)$ within a compact, convex set $\U$---the uncertainty
set, borrowing terminology from the robust-optimization literature---and
define the corresponding minimum and maximum optimal values $p^-$ and
$p^+$, respectively. We call our approach {\em robust sensitivity
analysis\/}, or {\em RSA\/}. Then, continuing in Section \ref{sec:rsa},
we formulate the calculation of $p^-$ and $p^+$ as nonconvex bilinear
QPs (or BQPs) and briefly discuss attainability and complexity issues.
We also discuss how $p^-$ and $p^+$ may be infinite and suggest
alternative bounded variants, $q^-$ and $q^+$, which have the property
that, if $p^-$ is already finite, then $q^- = p^-$ and similarly for
$q^+$ and $p^+$. Compared to related approaches in the literature, our
discussion of finiteness is unique. We then discuss the addition of
redundant constraints to the formulations of $q^-$ and $q^+$, which will
strengthen later relaxations. Section \ref{sec:relax} then establishes
COP reformulations of the nonconvex BQPs by directly applying existing
reformulation techniques. Then, based on the COPs, we develop tractable
SDP-based relaxations that incorporate the structure of the uncertainty
set $\U$, and we also discuss procedures for generating feasible
solutions of the BQPs, which can also be used to verify the quality
of the relaxation bounds. In Section \ref{sec:expr}, we validate our
approach on several examples, which demonstrate that the relaxations
provide effective approximations of $q^+$ and $q^-$. In fact, we find
that the relaxations admit no gap with $q^+$ and $q^-$ for all 
tested examples.

We mention some caveats about the paper. First, we focus only on how the
optimal value is affected by uncertainty, not the optimal solution. We
do so because we believe this will be a more feasible first endeavor;
determining how general perturbations affect the optimal solution
can certainly be a task for future research. Second, as mentioned
above, we believe we are the first to consider these types of general
perturbations, and thus the literature with which to compare is somewhat
limited. However, we connect with the literature whenever possible,
e.g., in special cases such as interval perturbations and worst-case
linear optimization. Third, since we do not make any distributional
assumptions about the uncertainty of the parameters, nor about their
independence or dependence, we believe our approach aligns well with the
general sprit of {\em robust optimization}. It is important to note,
however, that our interest is {\em not} robust optimization and is {\em
not} directly comparable to robust optimization. For example, while in
robust optimization one wishes to find a single optimal solution that
works well for all realizations of the uncertain parameters, here we
are only concerned with how the optimal value changes as the parameters
change. Finally, we note the existence of other relaxations for
nonconvex QPs including LP relaxations (see \cite{Sherali.Adams.1997})
and Lasserre-type SDP relaxations. Generally speaking, LP-based
relaxations are relatively weak (see \cite{Anstreicher1}); we do not
consider them in this paper. In addition, SDP approaches can often be
tailored to outperform the more general Lasserre approach as has been
demonstrated in \cite{Peng.Zhu.2015}. Our copostive- and SDP-based
approach is similar; see for example the valid inequalities discussed in
Section \ref{ssec:sdprelax}.

\subsection{Notation, terminology, and copositive optimization} \label{ssec:notation}

Let $\RR^n$ denote $n$-dimensional Euclidean space represented as
column vectors, and let $\RR_+^n$ denote the nonnegative orthant in
$\RR^n$. For a scalar $p \ge 1$, the $p$-norm of $v \in \RR^n$ is
defined $\|v\|_p := (\sum_{i=1}^n |v_i|^p)^{1/p}$, e.g., $\|v\|_1 =
\sum_{i=1}^n |v_i|$. We will drop the subscript for the $2$-norm, e.g.,
$\|v\| := \|v\|_2$. For $v,w \in \RR^n$, the inner product of $v$ and
$w$ is defined as $v^Tw = \sum_{i=1}^n v_iw_i$ and the Hadamard product
of $v$ and $w$ is defined by $v \circ w := (v_1w_1,..., v_nw_n)^T
\in \RR^n$. $\RR^{m \times n}$ denotes the set of real $m \times n$
matrices, and the trace inner product of two matrices $A, B \in \RR^{m
\times n}$ is defined $A\bullet B := \text{trace}(A^TB)$. $\SYM^n$
denotes the space of $n \times n$ symmetric matrices, and for $X \in
\SYM^n$, $X \succeq 0$ denotes that $X$ is positive semidefinite. In
addition, $\diag(X)$ denotes the vector containing the diagonal entries
of $X$.

We also make several definitions related to {\em copositive
programming\/}. The $n \times n$ {\em copositive cone\/} is
defined as
\[
    \COP(\RR_+^n) := \{ M \in \SYM^n: x^T M x \ge 0 \ \forall \ x \in \RR_+^n \},
\]
and its dual cone, the {\em completely positive cone\/}, is
\[
    \CP(\RR_+^n) := \{ X \in \SYM^n: X = \textstyle{\sum_k} x^k (x^k)^T, \ x^k \in \RR_+^n \},
\]
where the summation over $k$ is finite but its cardinality is
unspecified. The term {\em copositive programming\/} refers to linear
optimization over $\COP(\RR_+^n)$ or, via duality, linear optimization
over $\CP(\RR_+^n)$. A more general notion of copositive
programming is based on the following ideas. Let $\K \subseteq \RR^n$ be
a closed, convex cone, and define
\begin{align*}
    \COP(\K) &:= \{ M \in \SYM^n: x^T M x \ge 0 \ \forall \ x \in \K \}, \\
    \CP(\K) &:= \{ X \in \SYM^n: X = \textstyle{\sum_k} x^k (x^k)^T, \ x^k \in \K \}.
\end{align*}
Then {\em generalized copositive programming\/} is linear optimization
over $\COP(\K)$ and $\CP(\K)$ and is also sometimes called {\em
set-semidefinite optimization\/} \cite{Eichfelder.Jahn.2008}. In this
paper, we work with generalized copositive programming, although we will
use the shorter phrase {\em copositive programming\/} for convenience.

\section{Robust Sensitivity Analysis} \label{sec:rsa}

In this section, we introduce the concept of robust sensitivity analysis
of the optimal value of the linear program (\ref{Eq:StdPrimal}).
In particular, we define the best-case optimal value $p^-$ and the
worst-case optimal value $p^+$ over the uncertainty set $\U$, which
contains general perturbations in the objective coefficients $\hat c$
and the right-hand sides $\hat b$. We then propose nonconvex bilinear
QPs (BQPs) to compute $p^-$ and $p^+$. Next, we clarify when $p^-$ and
$p^+$ could be infinite and propose finite, closely related alternatives
$q^+$ and $q^-$, which can also be formulated as nonconvex BQPs.
Importantly, we prove that $q^-$ equals $p^-$ whenever $p^-$ is finite;
the analogous relationship is also proved for $q^+$ and $p^+$.

\subsection{The best- and worst-case optimal values} \label{sec:optvalfunc}

In the Introduction, we have described $\hat b$ and $\hat c$ as
parameters that could vary, a concept that we now formalize. Hereafter,
$(\hat b, \hat c)$ denotes the {\em nominal\/}, ``best guess'' parameter
values, and we let $(b,c)$ denote perturbations with respect to $(\hat
b, \hat c)$. In other words, the true data could be $(\hat b + b, \hat c
+ c)$, and we think of $b$ and $c$ as varying. We also denote the {\em
uncertainty set\/} containing all possible perturbations $(b,c)$ as $\U
\subseteq \RR^m \times \RR^n$. Throughout this paper, we assume the
following:

\begin{assumption} \label{assmp:a1}
$\U$ is compact and convex, and $\U$ contains $(0,0)$. 
\end{assumption}

Given $(b,c) \in \U$, we define the perturbed optimal value function at
$(b,c)$ as
\begin{equation} \label{equ:p(b,c)}
p(b, c) := \min \{(\hat c + c)^Tx : \hat Ax = \hat b + b, \ x \ge 0 \}.
\end{equation}
For example, $p(0,0)$ is the {\em nominal optimal value\/} of the {\em
nominal problem\/} based on the nominal parameters. The main idea of
robust sensitivity analysis is then to compute the infimum (best-case)
and supremum (worst-case) of all optimal values $p(b,c)$ over the
uncertainty set $\U$, i.e., to calculate
\begin{align}
p^- &:= \inf\{ p(b,c) : (b,c) \in \U \}, \label{equ:p-} \\
p^+ &:= \sup\{ p(b,c) : (b,c) \in \U \}. \label{equ:p+}
\end{align}

\noindent We illustrate $p^-$ and $p^+$ with a small example.

\begin{example} \label{Ex:ex1}
Consider the nominal LP
\begin{equation} \label{Eq:simple_example}
\begin{array}{ll} 
\min  &  x_1 + x_2 \\
\st &  x_1 +  x_2  =  2 \\
     & x_1, x_2 \ge 0
\end{array}
\end{equation}
and the uncertainty set
\[
\U :=
\left\{ (b,c) : \begin{array}{c} b_1 \in [-1, 1] \\  c_1 \in [-0.5, 0.5], \ c_2 =0 \end{array}
 \right\}.
\]
Note that the perturbed data $\hat b_1 + b_1$ and $\hat c_1 + c_1$
remain positive, while $\hat c_2 + c_2$ is constant. Thus, the minimum
optimal value $p^-$ occurs when $b_1$ and $c_1$ are minimal,
i.e., when $b_1 = -1$ and $c_1 = -0.5$. In this case, $p^- = 0.5$ at
the solution $(x_1,x_2) =(1,0)$. In a related manner, $p^+ = 3$ when
$b_1 = 1$ and $c_1 = 0.5$ at the point $(x_1,x_2) =(0,3)$. Actually, any
perturbation with $c_1 \in [0,0.5]$ and $b_1=1$ realizes the worst-case
value $p^+=3$. Figure \ref{Fig:two_vars} illustrates this example.
\end{example}

\begin{figure}
        \begin{subfigure}[b]{0.5\textwidth}
                \includegraphics[width=\linewidth]{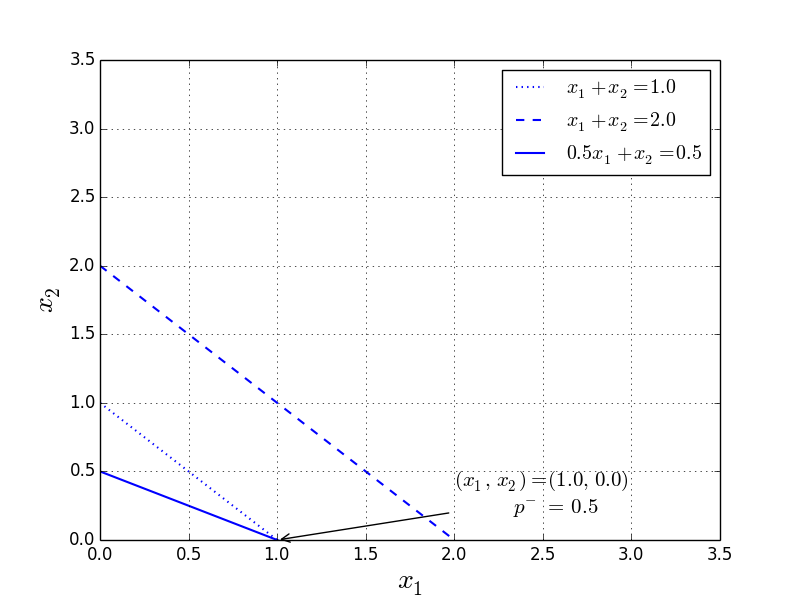}
                \caption{Illustration of the best-case optimal value}
                \label{Fig:p-}
        \end{subfigure}%
        \begin{subfigure}[b]{0.5\textwidth}
                \includegraphics[width=\linewidth]{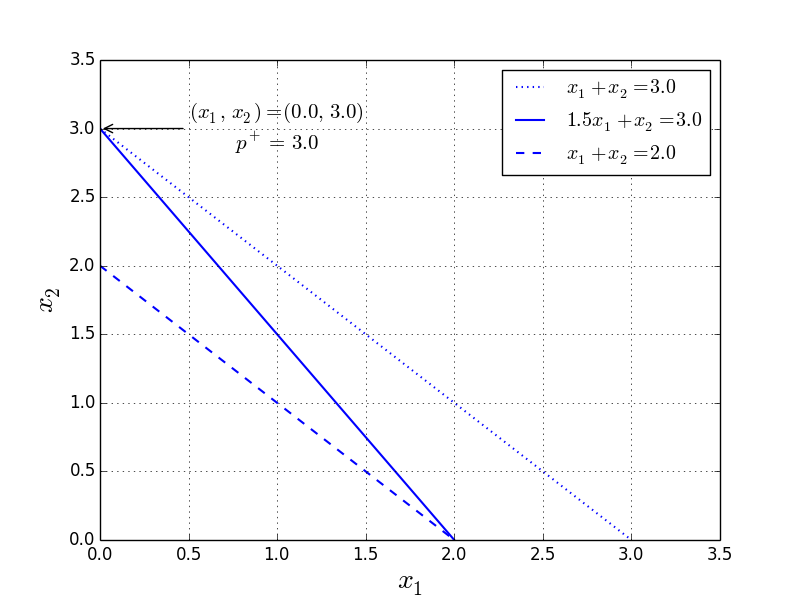}
                \caption{Illustration of the worst-case optimal value}
                \label{Fig:p+}
        \end{subfigure}%
        \caption{Illustration of Example \ref{Ex:ex1}. Note that
the dashed line in both (\ref{Fig:p-}) and (\ref{Fig:p+}) corresponds
to the feasible region of the nominal problem.}\label{Fig:two_vars}
\end{figure}

We can obtain a direct formulation of $p^-$ by simply collapsing the
inner and outer minimizations of (\ref{equ:p-}) into a single nonconvex
BQP:

\begin{equation} 
\begin{array}{lll} \label{equ:quadp-}
& \inf_{b,c,x} & (\hat c + c)^Tx  \\
& \st & \hat Ax = \hat b + b, \ x \ge 0 \\
& & (b,c) \in \U.
\end{array}
\end{equation}

\noindent The nonconvexity comes from the bilinear term $c^Tx$ in the
objective function. In the special case that $(b,c) \in \U$ implies $c =
0$, i.e., when there is no perturbation in the objective coefficients,
we have the following:

\begin{remark} \label{rem:polytime_c0}
If $\U$ is tractable and $c = 0$ for all $(b,c) \in \U$, then
$p^-$ can be computed in polynomial time as the optimal value of
(\ref{equ:quadp-}) with $c = 0$, which is a convex program.
\end{remark}

A direct formulation for $p^+$ can, under a fairly
weak assumption, be gotten via duality. Define the perturbed primal and
dual feasible sets for any $(b,c) \in \U$:
\begin{align*}
P(b) &:= \{ x : \hat A x = \hat b + b, x \ge 0 \}, \\
D(c) &:= \{ (y,s) \ge 0 : \hat A^T y + s = \hat c + c, s \ge 0  \}.
\end{align*}
For instance, $P(0)$ and $D(0)$ are the primal-dual feasible sets of the
nominal problem. Next define the dual LP for (\ref{equ:p(b,c)}) as
\begin{equation*} 
    d(b,c) := \max \{ (\hat b + b)^T y : (y,s) \in D(c) \}.
\end{equation*}
Considering the extended notion of strong duality, which handles the
cases of infinite values, we have that $d(b,c) = p(b,c)$ when at least
one of $P(b)$ and $D(c)$ is nonempty. Hence, under the assumption
that every $(b,c) \in \U$ yields $P(b) \ne \emptyset$ or $D(c) \ne
\emptyset$, a direct formulation for $p^+$ can be constructed by
replacing $p(b,c)$ in (\ref{equ:p-}) with $d(b,c)$ and then collapsing
the subsequent inner and outer maximizations into the single nonconvex
BQP
\begin{equation} 
\begin{array}{lll} \label{equ:quadp+}
& \sup_{b,c,y,s} & (\hat b + b)^Ty  \\
& \st & \hat A^Ty + s = \hat c + c, \ s \ge 0 \\
& & (b,c) \in \U.
\end{array}
\end{equation}
Here again, the nonconvexities arise due to the bilinear term $b^T y$
in the objective. If $(b,c) \in \U$ implies $b = 0$, then $p^+$ can be
calcuated in polynomial time:

\begin{remark} \label{rem:polytime_b0}
If $\U$ is tractable and $b = 0$ for all $(b,c) \in \U$, then
$p^+$ can be computed in polynomial time as the optimal value of
(\ref{equ:quadp+}) with $b = 0$, which is a convex program.
\end{remark}

We summarize the above discussion in the following proposition:

\begin{proposition} \label{pro:1stformulation}
The best-case value $p^-$ equals the optimal value of (\ref{equ:quadp-}).
Moreover, if $P(b) \ne \emptyset$ or $D(c) \ne \emptyset$ for all $(b,c)
\in \U$, then the worst-case value $p^+$ equals the optimal value of
(\ref{equ:quadp+}).
\end{proposition}

\noindent We view the condition in Proposition \ref{pro:1stformulation}---that
at least one of $P(b)$ and $D(c)$ is nonempty for each $(b,c) \in
\U$---to be rather mild. Said differently, the case that $P(b) = D(c) =
\emptyset$ for some $(b,c) \in \U$ appears somewhat pathological. For
practical purposes, we hence consider (\ref{equ:quadp+}) to be a valid
formulation of $p^+$. Actually, in the next subsection, we will further
restrict our attention to those $(b,c) \in \U$ for which both $P(b)$ and
$P(c)$ are nonempty. In such cases, each $p(b,c)$ is guaranteed to be
finite, which---as we will show---carefully handles the cases when $p^+$
and $p^-$ are infinite.

Indeed, the worst-case value $p^+$ could equal $+\infty$ due to some
perturbed $P(b)$ being empty as shown in the following example:

\begin{example} \label{Ex:ex2}
In Example \ref{Ex:ex1}, change the uncertainty set to
\[
\U :=
\left\{ (b,c) : \begin{array}{c} b_1 \in [-3, 1] \\  c_1 \in [-0.5, 0.5], c_2 =0 \end{array}
 \right\}.
\]
Then $p(b,c) = +\infty$ whenever $b_1 \in [-3, -2)$ since then the
primal feasible set $P(b)$ is empty. Then $p^+ = +\infty$ overall.
However, limiting $b_1$ to $[-2,1]$ yields a worst-case value of 3
as discussed in Example \ref{Ex:ex1}.
\end{example}

\noindent Similarly, $p^-$ might equal $-\infty$ due to some perturbed
LP having unbounded objective value, implying infeasibility of the
corresponding dual feasible set $D(c)$.

\subsection{Attainability and complexity} \label{ssec:complexity}

In this brief subsection, we mention results pertaining to the
attainability of $p^-$ and $p^+$ and the computational complexity of
computing them.

By an existing result concerning the attainability of the optimal value
of nonconvex BQPs, we have that $p^-$ and $p^+$ are attainable when $\U$
has a relatively simple structure:

\begin{proposition}[theorem 2 of \cite{Luo.Zhang.1999}]
Suppose $\U$ is representable by a finite number of linear constraints
and at most one convex quadratic constraint. Then, if the optimal value
of (\ref{equ:quadp-}) is finite, it is attained. A similar statement
holds for (\ref{equ:quadp+}).
\end{proposition}

\noindent In particular, attainability holds when $\U$ is polyhedral
or second-order-cone representable with at most one second-order cone.
Moreover, the bilinear nature of (\ref{equ:quadp-}) implies that, if the
optimal value is attained, then there exists an optimal solution $(x^*,
b^*, c^*)$ with $(b^*,c^*)$ an extreme point of $\U$. The same holds for
(\ref{equ:quadp+}) if its optimal value is attained.

As discussed in the Introduction, the worst-case value $p^+$ has been
studied by Peng and Zhu \cite{Peng.Zhu.2015} for the special case when
$c = 0$ and $b$ is contained in an ellipsoid. The authors demonstrate
(see their proposition 1.1) that calculating $p^+$ in this case is
NP-hard. By the symmetry of duality, it thus also holds that $p^-$ is
NP-hard to compute in general.

\subsection{Finite variants of $p^-$ and $p^+$} \label{ssec:finite}

We now discuss closely related variants of $p^+$ and $p^-$ that are
guaranteed to be finite and to equal $p^+$ and $p^-$, respectively, when
those values are themselves finite. We require the following feasibility
and boundedness assumption:

\begin{assumption} \label{assmp:a2}
Both feasible sets $P(0)$ and $D(0)$ are nonempty, and one is bounded.
\end{assumption}

By standard theory, $P(0)$ and $D(0)$ cannot both
be nonempty and bounded. Also define
\[
    \overline{\U} := \{ (b,c) \in \U : P(b) \ne \emptyset, D(c) \ne \emptyset \}.
\]
Note that $(0,0) \in \overline{\U}$ due to Assumption \ref{assmp:a2}.
In fact, $\overline{\U}$ can be captured with linear constraints that
enforce primal-dual feasibility and hence is a compact, convex subset of
$\U$:
\[
    \overline{\U} = \left\{ (b,c) \in \U :
        \begin{array}{c}
            %\exists \; (x,y,s) \text{ such that} \\
            \hat Ax = \hat b + b, x \ge 0 \\
            \hat A^T y + s = \hat c + c, s \ge 0 \\
        \end{array}
    \right\}.
\]
Analogous to $p^+$ and $p^-$, define
\begin{align}
    q^+ &:= \sup\{ p(b,c) : (b,c) \in \overline{\U} \} \label{equ:barp+} \\
    q^- &:= \inf\{ p(b,c) : (b,c) \in \overline{\U} \}. \label{equ:barp-}
\end{align}

\noindent The following proposition establishes the finiteness of $q^+$ and $q^-$:

\begin{proposition} \label{prop:finiteness}
Under Assumptions \ref{assmp:a1} and \ref{assmp:a2}, both $q^+$ and
$q^-$ are finite.
\end{proposition}

\begin{proof}
We prove the contrapositive for $q^-$. (The argument for $q^+$ is
similar.) Suppose $q^- = -\infty$. Then there exists a sequence $\{
(b^k,c^k) \} \subseteq \overline{\U}$ with finite optimal values
$p(b^k,c^k) \to -\infty$. By strong duality, there exists a primal-dual
solution sequence $\{ (x^k,y^k,s^k) \}$ with $(\hat c + c)^T x^k = (\hat
b + b)^T y^k \to -\infty$. Since $\overline{\U}$ is bounded, it follows
that $\|x^k\| \to \infty$ and $\| y^k \| \to \infty$.

Consider the sequence $\{ (z^k,d^k) \}$ with $(z^k,d^k) := (x^k,b^k)/
\|x^k\|.$ We have $\hat A z^k = \hat b/\|x^k\| + d^k$, $z^k \ge 0$,
and $\|z^k\| = 1$ for all $k$. Morover, $\hat b/\|x^k\| + d^k \to 0$.
Hence, there exists a subsequence converging to $(\bar z, 0)$ such
that $\hat A \bar z = 0$, $\bar z \ge 0$, and $\|\bar z\| = 1$. This
proves that the recession cone of $P(0)$ is nontrivial, and hence $P(0)$
is unbounded. In a similar manner, $D(0)$ is unbounded, which means
Assumption \ref{assmp:a2} does not hold.
\end{proof}

\noindent Note that the proof of Proposition \ref{prop:finiteness} only
assumes that $\U$, and hence $\overline{\U}$, is bounded, which does not
use the full power of Assumption \ref{assmp:a1}.

Similar to $p^-$, a direct formulation of $q^-$ can be constructed
by employing the primal-dual formulation of $\overline{\U}$ and by
collapsing the inner and outer minimizations of (\ref{equ:barp-}) into a
single nonconvex BQP:
\begin{align} \label{equ:quadbarp-}
\begin{array}{lll} 
    q^- \ = \ & \inf_{b,c,x,y,s} & (\hat c + c)^Tx  \\
 & \st & \hat Ax = \hat b + b, \ x \ge 0 \\
& & \hat A^Ty + s = \hat c + c, \ s \ge 0 \\
& & (b,c) \in \U.
\end{array}
\end{align}
\noindent Likewise for $p^+$, after replacing $p(b,c)$ in
(\ref{equ:barp+}) by $d(b,c)$, we can collapse the inner and outer
maximizations into a single nonconvex BQP:
\begin{align} \label{equ:quadbarp+}
\begin{array}{lll} 
    q^+ \ = \ & \sup_{b,c,x,y,s} & (\hat b + b)^Ty  \\
& \st & \hat Ax = \hat b + b, \ x \ge 0 \\
&& \hat A^Ty + s = \hat c + c, \ s \ge 0 \\
&& (b,c) \in \U.
\end{array}
\end{align}

The following proposition establishes $q^+ = p^+$ when $p^+$ is
finite and, similarly, $q^- = p^-$ when $p^-$ is finite.

\begin{proposition}
If $p^+$ is finite, then $q^+ = p^+$, and if $p^-$ is finite, then
$q^- = p^-$.
\end{proposition}

\begin{proof}
We prove the second statement only since the first is similar. Comparing
the formulation (\ref{equ:quadp-}) for $p^-$ and the formulation
(\ref{equ:quadbarp-}) for $q^-$, it is clear that $p^- \le
q^-$. In addition, let $(b,c,x)$ be any feasible solution of
(\ref{equ:quadp-}). Because $p^-$ is finite, $p(b,c)$ is finite. Then
the corresponding dual problem is feasible, which implies that we can
extend $(b,c,x)$ to a solution $(b,c,x,y,s)$ of (\ref{equ:quadbarp-})
with the same objective value. Hence, $p^- \ge q^-$.
\end{proof}

In the remaining sections of the paper, we will focus on the
finite variants $q^-$ and $q^+$ given by the nonconvex QPs
(\ref{equ:quadbarp-}) and (\ref{equ:quadbarp+}), which optimize the
optimal value function $p(b,c) = d(b,c)$ based on enforcing primal-dual
feasibility. It is clear that we may also enforce the complementary
slackness condition $x \circ s = 0$ without changing these problems.
Although it might seem counterintuitive to add the redundant, nonconvex
constraint $x \circ s = 0$ to an already difficult problem, in Section
\ref{sec:relax}, we will propose convex relaxations to approximate $q^-$
and $q^+$, in which case---as we will demonstrate---the relaxed versions
of the redundant constraint can strengthen the relaxations.

\section{Copositve Formulations and Relaxations} \label{sec:relax}

In this section, we use copositive optimization techniques
to reformulate the RSA problems (\ref{equ:quadbarp-}) and
(\ref{equ:quadbarp+}) into convex programs. We further relax the
copositive programs into conic, SDP-based problems, which are
computationally tractable.

\subsection{Copositive formulations} \label{sec:coprelax}

In order to formulate (\ref{equ:quadbarp-}) and (\ref{equ:quadbarp+})
as COPs, we apply a result of \cite{Burer.2011}; see also
\cite{Burer.2009,Dickinson.et.al.2013,Eichfelder.Povh.2013}. Consider
the general nonconvex QP
\begin{align}
    \inf \ \ \ \ &z^T W z + 2 \, w^T z \label{equ:genQP} \\
    \st \ \ \ &Ez = f, \ z \in \K \nonumber
\end{align}
where $\K$ is a closed, convex cone. Its copositive reformulation is
\begin{align}
    \inf \ \ \ \ &W \bullet Z + 2 \, w^T z \label{equ:genCOP} \\
    \st \ \ \ &Ez = f, \ \diag(EZE^T) = f \circ f \nonumber \\
    &\begin{pmatrix} 1 & z^T \\ z & Z \end{pmatrix} \in \CP(\RR_+ \times \K), \nonumber
\end{align}
as established by the following lemma:

\begin{lemma}[corollary 8.3 in \cite{Burer.2011}] \label{lemma:Burer}
Problem (\ref{equ:genQP}) is equivalent to (\ref{equ:genCOP}), i.e.:
(i) both share the same optimal value; (ii) if $(z^*,Z^*)$ is optimal
for (\ref{equ:genCOP}), then $z^*$ is in the convex hull of optimal
solutions for (\ref{equ:genQP}).
\end{lemma}

The following theorem establishes that problems (\ref{equ:quadbarp-})
and (\ref{equ:quadbarp+}) can be reformulated as copositive programs
according to Lemma \ref{lemma:Burer}. The proof is based on describing
how the two problems fit the form (\ref{equ:genQP}).

\begin{theorem} \label{the:cop}
Problems (\ref{equ:quadbarp-}) and (\ref{equ:quadbarp+}) to
compute $q^-$ and $q^+$ are solvable as copositive programs of the form
(\ref{equ:genCOP}), where
\[
    \K := \homg(\U) \times \RR_+^n \times \RR^m \times \RR^n_+
\]
and
\[
    \homg(\U) := \{ (t,b,c) \in \RR_+ \times \RR^m \times \RR^n :
    t > 0, \ (b,c)/t \in \U \} \cup \{ (0,0,0) \}
\]
is the homogenization of $\U$.
\end{theorem}

\begin{proof}
We prove the result for just problem (\ref{equ:quadbarp-}) since
the argument for problem (\ref{equ:quadbarp+}) is similar. First,
we identify $z \in \K$ in (\ref{equ:genQP}) with $(t,b,c,x,y,s)
\in \homg(\U) \times \RR_+^n \times \RR^m \times \RR_+^n$ in
(\ref{equ:quadbarp-}). In addition, in the constraints, we identify
$Ez = f$ with the equations $\hat A x = t \hat b + b$, $\hat A^T y +
s = t \hat c + c$, and $t = 1$. Note that the right-hand-side vector
$f$ is all zeros except for a single entry corresponding to the
constraint $t = 1$. Moreover, in the objective, $z^T W z$ is identified
with the bilinear term $c^T x$, and $2 \, w^T z$ is identified with
the linear term $\hat c^T x$. With this setup, it is clear that
(\ref{equ:quadbarp-}) is an instance of (\ref{equ:genQP}) and hence
Lemma \ref{lemma:Burer} applies to complete the proof.
\end{proof}

\subsection{SDP-based conic relaxations} \label{ssec:sdprelax}

As discussed above, the copositive formulations of
(\ref{equ:quadbarp-}) and (\ref{equ:quadbarp+}) as represented
by (\ref{equ:genCOP}) are convex yet generally intractable. Thus, we
propose SDP-based conic relaxations that are polynomial-time solvable
and hopefully quite tight in practice. In Section \ref{sec:expr} below,
we will investigate their tightness computationally.

We propose relaxations that are formed from (\ref{equ:genCOP}) by
relaxing the cone constraint
$$
M := \begin{pmatrix} 1 & z^T \\ z & Z \end{pmatrix} \in \CP(\RR_+ \times \K).
$$
As is well known---and direct from the definitions---cones of the form
$\CP(\cdot)$ are contained in the positive semidefinite cone. Hence, we
will enforce $M \succeq 0$. It is also true that $M \in \CP(\RR_+ \times
\K)$ implies $z \in \K$, although $M \succeq 0$ does not necessarily
imply this. So, in our relaxations, we will also enforce $z \in \K$.
Including $z \in \K$ improves the relaxation and also helps in the
calculation of bounds in Section \ref{ssec:boundsfeas}

Next, suppose that the description of $\RR_+ \times \K$ contains at
least two linear constraints, $a_1^Tz \le b_1$ and $a_2^Tz \le b_2$. By
multiplying $b_1 - a_1^T z$ and $b_2 - a_2^T z$, we obtain a valid,
yet redundant, quadratic constraint $b_1b_2 - b_1a_2^Tz - b_2a_1^Tz
+ a_1^Tzz^Ta_2 \ge 0$ for $\CP(\RR_+ \times \K)$. This quadratic
inequality can in turn be linearized in terms of $M$ as $b_1b_2 -
b_1a_2^Tz - b_2a_1^Tz + a_1^TZa_2 \ge 0$, which is valid for $\CP(\RR_+
\times \K)$. We add this linear inequality to our relaxation; it is
called an {\em RLT constraint\/} \cite{Sherali.Adams.1997}. In fact,
we add all such RLT constraints arising from all pairs of linear
constraints present in the description of $\RR_+ \times \K$.

When the description of $\RR_+ \times \K$ contains at least one
linear constraint $a_1^Tz \le b_1$ and one second-order-cone constraint
$\| d_2 - C_2^T z \| \le b_2 - a_2^T z$, where $d_2$ is a vector and
$C_2$ is a matrix, we will add a so-called {\em SOC-RLT constraint\/} to
our relaxation \cite{Burer4}. The constraint is derived by multiplying
the two constraints to obtain the valid quadratic second-order-cone
constraint
\[
    \| (b_1 - a_1^T z)(d_2 - C_2^T z) \| \le (b_1 - a_1^T z)(b_2 - a_2^T z).
\]
After linearization by $M$, we have the second-order-cone constraint
\[
    \| b_1 d_2 - d_2 a_1^T z - b_1 C_2^T z + C_2^T Z a_1 \|
    \le
    b_1b_2 - b_1a_2^Tz - b_2a_1^Tz + a_1^TZa_2.
\]

Finally, recall the redundant complementarity constraint $x \circ s =
0$ described at the end of Section \ref{ssec:finite}, which is valid for both
(\ref{equ:quadbarp-}) and (\ref{equ:quadbarp+}). Decomposing it as $x_i
s_i = 0$ for $i = 1,\ldots,n$, we may translate these $n$ constraints to
(\ref{equ:genCOP}) as $z^T H_i z = 0$ for appropriatly defined matrices
matrices $H_i$. Then they may be linearized and added to our
relaxation as $H_i \bullet Z = 0$.

To summarize, let $\RLT$ denote the set of $(z, Z)$ satisfying all the
derived RLT constraints, and similarly, define $\SOCRLT$ as the set
of $(z, Z)$ satisfying all the derived SOC-RLT constraints. Then the
SDP-based conic relaxation for (\ref{equ:genCOP}) that we propose to
solve is

\begin{equation}  \label{equ:genSDP} 
\begin{array}{ll}
 \inf   \ \ \ &W \bullet Z + 2 \, w^T z \\
   \st \ \ \ &Ez = f, \ \diag(EZE^T) = f \circ f  \\
   &H_i \bullet Z = 0  \ \ \forall \ i = 1,\ldots, n \\
   & (z, Z) \in \RLT \cap \SOCRLT \\
   &\begin{pmatrix} 1 & z^T \\ z & Z \end{pmatrix} \succeq 0, \ z \in \K.
\end{array}
\end{equation}

\noindent It is worth mentioning that, in many cases, the RLT and
SOC-RLT constraints will already imply $z \in \K$, but in such cases, we
nevertheless write the constraint in (\ref{equ:genSDP}) for emphasis;
see also Section \ref{ssec:boundsfeas} below.

When translated to the problem (\ref{equ:quadbarp-}) for
calculating $q^-$, the relaxation (\ref{equ:genSDP}) gives rise to
a lower bound $q_{\SDP}^- \le q^-$. Similarly, when applied to
(\ref{equ:quadbarp+}), we get an upper bound $q_{\SDP}^+ \ge q^+$.

\subsection{Bounds from feasible solutions} \label{ssec:boundsfeas}

In this section, we discuss two methods to approximate $q^-$
from above and $q^+$ from below, i.e., to bound $q^-$ and
$q^+$ using feasible solutions of (\ref{equ:quadbarp-}) and
(\ref{equ:quadbarp+}), respectively.

The first method, which has been inspired by \cite{Peng.Zhu.2015},
utilizes the optimal solution of the SDP relaxation (\ref{equ:genSDP}).
Let us discuss how to obtain such a bound for (\ref{equ:quadbarp-}),
as the discussion for (\ref{equ:quadbarp+}) is similar. We first
observe that any feasible solution $(z,Z)$ of (\ref{equ:genSDP})
satisfies $Ez = f$ and $z \in \K$, i.e., $z$ satisfies all of the
constraints of (\ref{equ:genQP}). Since (\ref{equ:genQP}) is equivalent
to (\ref{equ:quadbarp-}) under the translation discussed in the proof
of Theorem \ref{the:cop}, $z$ gives rise to a feasible
solution $(x,y,s,b,c)$ of (\ref{equ:quadbarp-}). From this feasible
solution, we can calculate $(\hat c + c)^T x \ge q^-$. In practice, we
will start from the optimal solution $(z^-, Z^-)$ of (\ref{equ:genSDP}).
We summarize this approach in the following remark.

\begin{remark} \label{rem:bound}
Suppose that $(z^-, Z^-)$ is an optimal solution of the relaxation
(\ref{equ:genSDP}) corresponding to (\ref{equ:quadbarp-}), and let
$(x^-,y^-,s^-,b^-,c^-)$ be the translation of $z^-$ to a feasible point
of (\ref{equ:quadbarp-}). Then, $r^- := (\hat c + c^-)^Tx^- \ge q^-$.
Similarly, we define $r^+ := (\hat b + b^+)^T y^+ \le q^+$ based on an
optimal solution $(z^+, Z^+)$ of (\ref{equ:genSDP}) corresponding to
(\ref{equ:quadbarp+}).
\end{remark}

Our second method for bounding $q^-$ and $q^+$ using feasible solutions
is a sampling procedure detailed in Algorithm \ref{Algo:A1}. The main
idea is to generate randomly a point $(b,c) \in \overline{\U}$ and then
to calculate $p(b,c)$, which serves as an upper bound of $p^-$ and a
lower bound of $p^+$, i.e., $p^- \le p(b,c) \le p^+$. Multiple points
$(b^k,c^k)$ and values $p^k := p(b^k,c^k)$ are generated and the best
bounds $p^- \le v^- := \min_k \{ p^k \}$ and $\max_k \{ p^k \} =: v^+
\le p^+$ are saved. In fact, by the bilinearity of (\ref{equ:quadbarp-})
and (\ref{equ:quadbarp+}), we we may restrict attention to the
extreme points $(b,c)$ of $\overline{\U}$ without reducing the
quality of the resultant bounds; see also the discussion in Section
\ref{ssec:complexity}. Hence, Algorithm \ref{Algo:A1} generates---with
high probability---a random extreme point of $\overline{\U}$ by
optimizing a random linear objective over $\overline{\U}$, and we
generate the random linear objective as a vector uniform on the sphere,
which is implemented by a well-known, quick procedure. Note that,
even though the random objective is generated according to a specific
distribution, we cannot predict the resulting distribution over the
extreme points of $\overline{\U}$.

\begin{algorithm}
\caption{Sampling procedure to bound $q^-$ from above and $q^+$ from below}\label{Algo:A1}
\begin{algorithmic}
\STATE \textbf{Inputs:} Instance with uncertainty set $\U$ and restricted
uncertainty set $\overline{\U}$. Number of random trials $T$.
\STATE \textbf{Outputs:} Bounds $v^- := \min_k \{ p^k \} \ge p^-$ and $v^+ := \max_k \{ p^k \}
\le p^+$.
\FOR{$k=1,\ldots,T$}
\STATE Generate $(f,g) \in \RR^m \times \RR^n$ uniformly on the unit sphere.
\STATE Calculate $(b^k,c^k) \in \Arg\min \{ f^T b + g^T c : (b,c) \in \overline{\U} \}$.
\STATE Set $p^k := p(b^k,c^k)$.
\ENDFOR
\end{algorithmic}
\end{algorithm}

As all four of the bounds $r^-, r^+, v^-$, and $v^+$ are constructed
from feasible solutions, we can further improve them heuristically by
exploiting the bilinear objective functions in (\ref{equ:quadbarp-})
and (\ref{equ:quadbarp+}). In particular, we employ the standard local
improvement heuristic for programs with a bilinear objective and
convex constraints (e.g., see \cite{Konno}). Suppose, for example,
that we have a feasible point $(x^-, y^-, s^-, b^-, c^-)$ for problem
(\ref{equ:quadbarp-}) as discussed in Remark \ref{rem:bound}.
To attempt to improve the solution, we fix the variable $c$ in
(\ref{equ:quadbarp-}) at the value $c^-$, and we solve the resulting
convex problem for a new, hopefully better point $(x^1, y^1, s^1, b^1,
c^1)$, where $c^1 = c^-$. Then, we fix $x$ to $x^1$, resolve, and
get a new point $(x^2, y^2, s^2, b^2, c^2)$, where $x^2 = x^1$. This
alternating process is repeated until there is no further improvement in the
objective of (\ref{equ:quadbarp-}), and the final objective is our bound
$r^-$.
 
In Section \ref{sec:expr} below, we use the bounds $r^-$, $r^+$,
$v^-$, and $v^+$ to verify the quality of our bounds $q_{\SDP}^-$ and
$q_{\SDP}^+$. Our tests indicate that neither bound, $r^-$ nor $v^-$,
dominates the other---and similarly for the bounds $r^+$ and $v^+$.
Hence, we will actually report the better of each pair: $\min\{r^-,
v^-\}$ and $\max\{r^+, v^+ \}$. Also, for the calculations of $v^-$ and
$v^+$, we always take $T = 10,000$ in Algorithm \ref{Algo:A1}.

\section{Computational Experiments} \label{sec:expr}

In this section, we validate our approach by testing it on six examples
from the literature as well as an example of our own making. The first
three examples in Section \ref{ssec:sa_literature} correspond to
classical sensitivity analysis approaches for LP; the fourth example
in Section \ref{ssec:inventory} corresponds to an interval LP in
inventory management; the fifth example in Section \ref{ssec:sysrisk}
corresponds to a systemic-risk calculation in financial systems; and
the last example in Section \ref{ssec:networkflow} is a transportation
network flow problem. We implement our tests in Python (version 2.7.6)
with Mosek (version 7.1.0.33) as our convex-optimization solver. All
of Mosek's settings are set at their defaults, and computations are
conducted on a Macintosh OS X Yosemite system with a quad-core 3.20GHz
Intel Core i5 CPU and 8 GB RAM.

\subsection{Examples from classical sensitivity analysis} \label{ssec:sa_literature}

Consider the following nominal problem from \cite{Wendell4}:
\[
\begin{array}{llll} 
\min  &  -12 x_1 - 18 x_2 - 18 x_3 - 40 x_4  & &\\
\st & \quad \ 4 x_1 + 9 x_2 + 7 x_3 + 10 x_4 + x_5 & =  &6000 \\
     & \quad \ \ \ x_1 + \ x_2 + 3 x_3 + 40 x_4 \ \ \ \ \ \ \ + x_6 & = & 4000 \\
     & \quad \ \ \ x_1, \ldots, x_6 \ge 0.
\end{array}
\]
The optimal basis is $B = \{1, 4\}$ with optimal solution $\tfrac13
(4000, 0, 0, 200, 0, 0)$ and optimal value $p(0,0)=-18667$. According to
standard, ``textbook'' sensitivity analysis, the optimal basis persists
when the coefficient of $x_1$ lies in the interval $[-16, -10]$ and
other parameters remain the same. Along this interval, one can easily
compute the best-case value $-24000$ and worst-case value $-16000$,
and we attempt to reproduce this analysis with our approach. So let us
choose the uncertainty set
\[
\U =
\left\{ (b, c) \in \RR^2 \times \RR^6: 
\begin{array}{c}
   b_1 = b_2 = 0  \\
   c_1 \in [-4,2] \\
   c_2 =  \dots = c_6 = 0
\end{array} \right\},
\]
which corresponds precisely to the above allowable decrease and increase
on the coefficient of $x_1$. Note that Assumptions \ref{assmp:a1}
and \ref{assmp:a2} are satisfied. We thus know from above that
$q^- = -24000$ and $q^+ = -16000$. Since $b = 0$ in $\U$, Remark
\ref{rem:polytime_b0} implies that $q^+$ is easy to calculate. So
we apply our approach, i.e., solving the SDP-based relaxation, to
approximate $q^-$. The relaxation value is $q_{\SDP}^-=-24000$, which
recovers $q^-$ exactly. The CPU time for computing $q_{\SDP}^-$ is 0.10
seconds.

Our second example is also based on the same nominal problem from
\cite{Wendell4}, but we consider the 100\%-rule. 
Again, we know that the optimal basis $B=\{1,4\}$ persists
when the coefficient of $x_1$ lies in the interval $[-16, -10]$ (and all
other parameters remain the same) or separately when the coefficient of
$x_2$ lies in the interval $[-134/3, +\infty]$ (and all other parameters
remain the same). In accordance with the 100\%-rule, we choose to
decrease the coefficient of $x_1$, and thus its allowed interval
is $[-16,-12]$ of width 4. We also choose to decrease the
coefficient of $x_2$, and thus its allowed interval is $[-134/3,
-18]$ of width $80/3$. The 100\%-rule ensures that the
optimal basis persists as long as the sum of fractions, corresponding to
the percent of maximum changes in the coefficients of $x_1$ and $x_2$,
is less than or equal to 1. In other words, suppose that $\tilde c_1$
and $\tilde c_2$ are the perturbed values of the coefficients of $x_1$
and $x_2$, respectively, and that all other coefficients stay the same.
Then the nominal optimal basis persists for $(\tilde c_1, \tilde c_2)$
in the following simplex:
\[
\left\{ (\tilde c_1, \tilde c_2) : 
\begin{array}{c}
    \tilde c_1 \in [-16, -12] \\
    \tilde c_2 \in [-134/3,-18] \\
  \tfrac{-12 - \tilde c_1}{4} + \tfrac{-18 - \tilde c_2}{80/3} \le 1
\end{array} \right\}.
\] 
By evaluating the three extreme points $(-12, -18)$, $(-16,-18)$
and $(-12, -134/3)$ of this set with respect to the nominal optimal
solution, one can calculate the best-case optimal value as $q^- =
-24000$ and the worst-case optimal value as $q^+ = -18667$. We again
apply our approach in an attempt to recover empirically the 100\%-rule.
Specifically, let
\[
\U =
\left\{ (b, c) : 
\begin{array}{c}
   b_1 = b_2 = 0 \\
   c_1 \in [-4,0], \ c_2 \in [-\frac{80}{3},0] \\
  - \tfrac{c_1}{4} - \tfrac{c_2}{80/3} \le 1 \\
   c_3 =  \dots = c_6 = 0
\end{array} \right\}.
\]
Note that Assumptions \ref{assmp:a1} and \ref{assmp:a2} are satisfied.
Due to $b=0$ and Remark \ref{rem:polytime_b0}, we focus our attention
on $q^-$. Calculating the SDP-based relaxation value, we see that
$q_{\SDP}^-=-24000$, which recovers $q^-$ precisely. The CPU time is
0.15 seconds

Our third example illustrates the tolerance approach, and we continue
to use the same nominal problem from \cite{Wendell4}. As mentioned in
the Introduction, the tolerance approach considers simultaneous and
independent perturbations in the objective coefficients by calculating a
maximum tolerance percentage such that, as long as selected coefficients
are accurate to within that percentage of their nominal values, the
nominal optimal basis persists; see \cite{Wendell3}. Let us consider
perturbations in the coefficients of $x_1$ and $x_2$ with respect to the
nominal problem. Applying the tolerance approach of \cite{Wendell3}, the
maximum tolerance percentage is $1/6$ in this case. That is, as long
as the two coefficient values vary within $-12 \pm 12/6 = [-14, -10]$
and $-18 \pm 18/6 = [-21,-15]$, respectively, then the nominal optimal
basis $B = \{1, 4\}$ persists. By testing the four extreme points of the
box of changes $[-14,-10] \times [-21,-15]$ with respect to the optimal
nominal solution, one can calculate the best-case optimal value as $q^-
= -21333$ and the worst-case optimal value as $q^+ = -16000$. To test
our approach in this setting, we set
\[
\U :=
\left\{ (b, c) : 
\begin{array}{c}
   b_1 = b_2 = 0, \ c_3 =  \dots = c_6 = 0 \\
   c_1 \in [-2,2], c_2 \in [-3,3] 
\end{array} \right\}
\]
and, as in the previous two examples, we focus on $q^-$. Assumptions
\ref{assmp:a1} and \ref{assmp:a2} are again satisfied, and we calculate
the lower bound $q_{\SDP}^-=-21333$, which recovers $q^-$ precisely. The
CPU time for computing $q_{\SDP}^-$ is 0.13 seconds.

\subsection{An example from interval linear programming} \label{ssec:inventory}

We consider an optimization problem that is typical in inventory
management, and this particular example originates from \cite{Gabrel}.
Suppose one must decide the quantity to be ordered during each period
of a finite, discrete horizon consisting of $T$ periods. The goal
is to satisfy exogenous demands $d_k$ for each period $k$, while
simultaneously minimizing the total of purchasing, holding, and shortage
costs. Introduce the following variables for each period $k$:
\begin{align*}
s_k = \ & \text{stock available at the end of period $k$;} \\
x_k = \ & \text{quantity ordered at the beginning of period $k$}.
\end{align*}
Items ordered at the beginning of period $k$ are delivered in time to
satisfy demand during the same period. Any excess demand is backlogged.
Hence, each $x_k$ is nonnegative, each $s_k$ is free, and
$$
s_{k-1} + x_k - s_k = d_k.
$$ 
The order quantities $x_k$ are further subject to uniform upper and
lower bounds, $u$ and $l$, and every stock level $s_k$ is bounded above
by $U$. At time $k$, the purchase cost is denoted as $c_k$, the holding
cost is denoted as $h_k$, and the shortage cost is denoted $g_k$. Then,
the problem can be formulated as the following linear programming
problem (assuming that the initial inventory is 0):
\begin{equation} 
\begin{array}{lll} \label{equ:ilp}
\min & \sum_{k=1}^T (c_kx_x + y_k) & \\
\st & s_0 = 0 \\
&  s_{k-1} + x_k - s_k = d_k & \ k=1,\dots,T  \\
& y_k \ge h_k s_k & \  k=1,\dots,T \\
&  y_k \ge -g_k s_k  & \ k=1,\dots,T \\
& l \le x_k \le u & \ k=1,\dots,T \\
& s_k \le U & \ k=1,\dots,T \\
& x_k, y_k \ge 0 & \ k=1,\dots,T. 
\end{array}
\end{equation}

As in \cite{Gabrel}, consider an instance of (\ref{equ:ilp}) in which
$T = 4$, $u=1500, l=1000,$ $U = 600$, and all costs are as in Table
\ref{Tab:inv_costs}. Moreover, suppose the demands $d_k$ are each
uncertain and may be estimated by the intervals $d_1 \in [700, 900], d_2
\in [1300, 1600], d_3 \in [900, 1100],$ and $d_4 \in [500, 700]$. From
\cite{Gabrel}, the worst-case optimal value over this uncertainty set is
$q^+ = 25600$. For our approach, it is easy to verify that Assumptions
\ref{assmp:a1} and \ref{assmp:a2} are satisfied, and solving
our SDP-based conic relaxation with an uncertainty set corresponding to the
intervals on $d_k$, we recover $q^+$ exactly, i.e., we have $q_{\SDP}^+
= 25600$. The CPU time for computing our SDP optimal value is 1,542
seconds.

\begin{table}[tbp]
\centering
\begin{tabular}{ccccc}
Period ($k$) & Purchasing cost ($c_k$) & Holding cost ($h_k$) & Shortage cost ($g_k$) \\ \hline
1 & 7   & 2  & 3 \\
2 & 1   & 1  & 4 \\
3 & 10 & 1 &  3 \\ 
4 & 6   & 1 &  3
\end{tabular}
\caption{Costs for each period of an instance of the inventory
management problem}
\label{Tab:inv_costs}
\end{table}

Since the uncertainties only involve the right-hand sides, Remark
\ref{rem:polytime_c0} implies that the best-case value $q^-$ can
be calculated in polynomial-time by solving an LP that directly
incorporates the uncertainty.

\subsection{Worst-case linear optimization} \label{ssec:sysrisk}

We next consider an example for calculating systemic risk in financial
systems, which is an application of worst-case linear optimization
(WCLO) presented in \cite{Peng.Zhu.2015}.

For an interbank market, systemic risk is used to evaluate the potential
loss of the whole market as a response to the decisions made by the
individual banks \cite{Eisenberg.Noe.2001, Peng.Zhu.2015}. Specifically,
let us consider a market consisting of $n$ banks. We use an $n \times
n$ matrix $\hat L$ to denote the liability relationship between any
two banks in the market. For instance, the element $\hat L_{ij}$
represents the liability of bank $i$ to bank $j$. In the market, banks
can also receive exogenous operating cash flows to compensate their
potential shortfalls on incoming cash flows. We use $\hat b_i$ to
denote the exogenous operating cash flow received by bank $i$. Given
the vector $\hat b$, we calculate the systemic loss $l(\hat b)$ of
the market, which measures the amount of overall failed liabilities
\cite{Peng.Zhu.2015}:
\[
\begin{array}{lll} 
l(\hat b) = & \displaystyle{\min_x} & \sum_{i=1}^n (1 - x_i)  \\
&\st & (\sum_{j=1}^n\hat L_{ij})x_i - \sum_{j=1}^n\hat L_{ji}x_j \le \hat b_i \ \ \ \forall \; i = 1, \ldots, n  \\
& & x_i \le 1 \hspace*{1.92in} \forall \; i = 1, \ldots, n.
\end{array}
\]
Here the decision variables $x_i$ represent the ratio of the total
dollar payment by bank $i$ to the total obligation of bank $i$. These
ratios are naturally less than or equal to 1 ($x_i \le 1$) as the banks
do not pay more than their obligations. In contrast, $1 - x_i$ denotes
the ratio of bank $i$ failing to fulfill its obligations. Furthermore,
we have a less-than-or-equal-to sign in the first constraint as the
system allows {\em limited liability } (see \cite{Eisenberg.Noe.2001}).
Finally, the objective is to minimize the total failure ratio of the
system.

In practice, however, there exist uncertainties in the exogenous
operating cash flows. Allowing for uncertainties, the worst-case
systemic risk problem \cite{Peng.Zhu.2015} is given as
\[
\begin{array}{lll} 
\displaystyle{\max_{b \in \V}} & \displaystyle{\min_x} & \sum_{i=1}^n (1 - x_i)  \\
&\st & (\sum_{j=1}^n\hat L_{ij})x_i - \sum_{j=1}^n\hat L_{ji}x_j \le  \hat b_i + Q_{i\cdot} b \ \ \ \forall \; i = 1, \ldots, n  \\
& & x_i \le 1 \hspace*{2.35in} \forall \; i = 1, \ldots, n.
\end{array}
\]
where $\V := \{b \in \RR^m: \|b\| \le 1\}$ denotes the uncertainty set,
$Q \in \RR^{n\times m}$ for some $m \le n$ corresponds to an affine
scaling of $\V$, and $Q_{i\cdot}$ denotes the $i$-th row of $Q$.
After converting the nominal LP to our standard form, we can easily
put the systemic risk problem into our framework by defining
$\U := \{ (b, c) : b \in \V, \ c = 0 \}$ and slightly changing our
$\overline \U$ to reflect the dependence on the affine transformation as
represented by the matrix $Q$.

\begin{table}[tbp]
\centering
\begin{tabular}{c|ccc|c}  % repeats {c|} 8 times
 Instance & $\max\{r^+,v^+\}$ & gap$^+$ & $q_{\SDP}^+$ & $t_{\SDP}^+ (s)$  \\ \hline
 1 &  0.521 & 0.0\%  & 0.521 & 0.66  \\ 
 2 &  1.429 & 0.0\%  & 1.429 & 0.82\\ 
 3 &  1.577 & 0.0\%  & 1.577 & 0.83\\ 
 4 &  0.344 & 0.0\%  & 0.344 & 0.68\\ 
 5 &  0.724 & 0.0\%  & 0.724 & 0.73\\ 
 6 &  0.690 & 0.0\%  & 0.690 & 0.80\\ 
 7 &  1.100 & 0.0\%  & 1.100 & 0.73\\ 
 8 &  0.458 & 0.0\%  & 0.458 & 0.78\\ 
 9 &  0.427 & 0.0\%  & 0.427 & 0.74\\ 
 10 &  0.300 & 0.0\%  & 0.300 & 0.78
\end{tabular}
\caption{Results for the systemic-risk example}
\label{Tab:sys_risk}
\end{table}

Similar to table 4 in \cite{Peng.Zhu.2015}, we randomly generate
10 instances of size $m \times n = 3 \times 5$. In accordance with
Remark \ref{rem:polytime_c0}, which states that $q^-$ is easy to
calculate in this case, we focus our attention on the worst-case value
$q^+$. It is straightforward to verify Assumptions \ref{assmp:a1} and
\ref{assmp:a2}. In Table \ref{Tab:sys_risk}, we list our 10 upper
bounds\footnote{Mosek encountered numerical problems
on some---but not all---of the generated system-risk instances. In
Table \ref{Tab:sys_risk}, we show 10 instances on which Mosek had
no numerical issues. From private communication with the Mosek
developers, it appears that the upcoming version of Mosek (version 8)
will have fewer numerical issues on these instances.} (one
for each of the 10 instances) in the column titled $q_{\SDP}^+$, and we
report the computation time (in seconds) for all 10 instances under the
column marked $t_{\SDP}^+$. To evaluate the quality of $q_{\SDP}^+$, we
also calculate $\max\{ r^+, v^+ \}$ for each instance and the associated
relative gap:
\[
\text{gap}^+ = \frac{q_{\SDP}^+ - \max\{ r^+, v^+\}}{\max\{ |\max\{ r^+, v^+\}|, \ 1 \}} \times 100\%.
\]
The computation times for computing $r^-$ and $r^+$ are trivial, while
the average computation time for computing $v^-$ and $v^+$ is about
$77$ seconds.

From the results in Table \ref{Tab:sys_risk}, we see clearly that our approach
recovers $q^+$ for all 10 instances, which also matches the quality of results
from \cite{Peng.Zhu.2015}.

\subsection{A network flow problem} \label{ssec:networkflow}

Next we consider a transportation network flow problem from
\cite{Xie}, which has $m_1=5$ suppliers/origins and $m_2=10$
customers/destinations for a total of $m=15$ facilities. The 
network is bipartite and consists of $n=24$ arcs connecting suppliers and customers; see
Figure \ref{Fig:network}. Also shown in Figure \ref{Fig:network} are the
(estimated) supply and demand numbers ($\hat b$) for each supplier and customer.
In addition, the (estimated) unit transportation costs ($\hat c$) associated with
the arcs of the network are given in Table \ref{Tab:cost}. Suppose at
the early stages of planning, the supply and demand units and the unit
transportation costs are uncertain. Thus, the manager would like to
quantify the resulting uncertainty in the optimal transportation cost.

\begin{figure}[ht]
  \centering
  \includegraphics[width=2.5in]{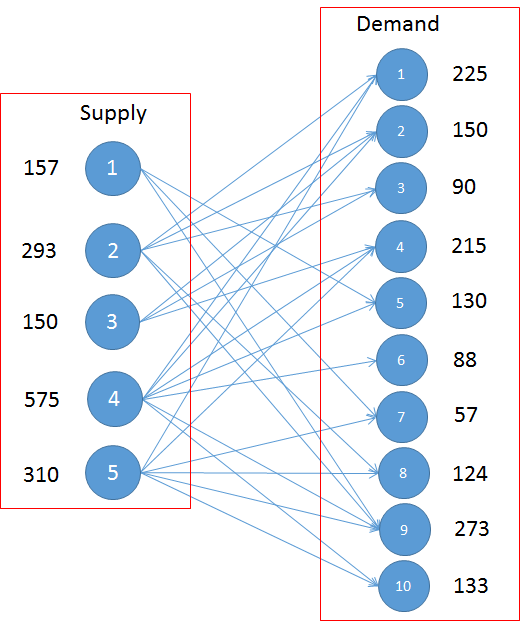}
  \caption{The transportation network of the 5 suppliers and 10 customers.}
  \label{Fig:network}
\end{figure}

\begin{table}[tbp]
\centering
\begin{tabular}{c|cccccccccc}  % repeats {c|} 4 times
& \multicolumn{10}{c}{Customer} \\
Supplier & 1 & 2 & 3 & 4 & 5 & 6 & 7 & 8 & 9 & 10  \\ \hline
1 & & & & & 2 & & 3 & & 2 & \\ 
2 & 3 & 3 & 4 & & & & & 1 & 4 & \\
3 & & 4 & 5 & 3 & & & & & & \\ 
4 & 1 & 2 & & 1 & 3 & 1 & & & 8 &  2 \\
5 & 4 & & & 3 & & & 2 & 1 & 2 & 1
\end{tabular}
\caption{The unit transportation costs associated with the arcs of the network.}
\label{Tab:cost}
\end{table}

%{\color{red} Are the inequalities in red correct, or are they left over
%from some previous discussions?}
We consider three cases for the uncertainty set, each of which is
 also parameterized by a scalar $\gamma \in (0,1)$. In the first case
 (``POLY''), we consider the polyhedral uncertainty set
 \[
 \U_1(\gamma) =
 \{ (b, c) : 
 \begin{array}{c}
     \|b\|_1 \le \gamma \|\hat b\|_1,  \  \|c\|_1 \le \gamma \|\hat c\|_1
 \end{array} \};
 \]
in the second case (``SOC''), we consider the second-order-cone
 uncertainty set
 \[
 \U_2(\gamma) :=
 \{ (b, c) : 
 \begin{array}{c}
   \|b\| \le \gamma \|\hat b\|, \ \|c\| \le \gamma \|\hat c\|
 \end{array} \};
 \]
 and in the third case (``MIX''), we consider a mixture of the first two
 cases:
 \[
 \U_3(\gamma) :=
 \{ (b, c) : 
 \begin{array}{c}
 \|b\|_1 \le \gamma \|\hat b\|_1,  \  \|c\| \le \gamma \|\hat c\|
 \end{array} \}.
 \]

\noindent For each, $ \gamma$ controls the perturbation magnitude in $b$
and $c$ relative to $\hat b$ and $\hat c$, respectively. In particular,
we will consider three choices of $\gamma$: 0.01, 0.03, and 0.05. For
example, $\gamma = 0.03$ roughly means that $b$ can vary up to 3\% of
the magnitude of $\hat b$. In total, we have three cases with three
choices for $\gamma$ resulting in nine overall experiments. 

Assumptions \ref{assmp:a1} and \ref{assmp:a2} are satisfied in this
example, and so we apply our approach to bound $q^-$ and $q^+$; see
Table \ref{Tab:network_reslt}. Our 18 bounds (lower and upper bounds
for each of the nine experiments) are listed in the two columns
titled $q_{\SDP}^-$ and $q_{\SDP}^+$, respectively. We also report
the computation times (in seconds) for all 18 instances under the two
columns marked $t_{\SDP}^-$ and $t_{\SDP}^+$. We also compute $r^-$,
$v^-$, $r^+$, and $v^+$ and define the relative gaps
\begin{align*}
\text{gap}^- &=
\frac{\min\{ r^-, v^-\} - q_{\SDP}^-}{\max\{ |\min\{ r^-, v^-\}|, \ 1 \}} \times 100\%, \\
\text{gap}^+ &= \frac{q_{\SDP}^+ - \max\{ r^+, v^+\}}{\max\{ |\max\{ r^+, v^+\}|, \ 1 \}} \times 100\%.
\end{align*}
Again, the computation times for $r^-$ and $r^+$ are trivial. The
average computation time for computing $v^-$ and $v^+$ is about $216$
seconds. 

Table \ref{Tab:network_reslt} shows that our relaxations capture $q^-$
and $q^-$ in all cases. As ours is the first approach to study general
perturbations in the literature, we are aware of no existing methods for
this problem with which to compare our results.

\begin{table}[tbp]
\centering
\begin{tabular}{cc|ccc|ccc|cc}  % repeats {c|} 8 times
Case                   & $\gamma$ & $q_{\SDP}^-$ & gap$^-$ & $ \min\{r^-, v^-\}$ & $ \max\{r^+, v^+\}$ & gap$^+$ & $q_{\SDP}^+$ & $t_{\SDP}^- (s)$ & $t_{\SDP}^+ (s)$ \\ \hline \hline
\multirow{3}{*}{POLY}  & 0.01     & 2638.4       & 0.0\%   & 2638.4              & 3088.8              & 0.0\%   & 3088.8       & 3551             & 4259 \\
                       & 0.03     & 2139.6       & 0.0\%   & 2139.6              & 3437.4              & 0.0\%   & 3437.4       & 5106             & 4433 \\
                       & 0.05     & 1640.9       & 0.0\%   & 1640.9              & 3769.8              & 0.0\%   & 3769.8       & 5330             & 4244 \\ \hline
\multirow{3}{*}{SOC}   & 0.01     & 2745.6       & 0.0\%   & 2745.6              & 2981.6              & 0.0\%   & 2981.6       & 190              & 122 \\
                       & 0.03     & 2498.9       & 0.2\%   & 2504.3              & 3212.1              & 0.0\%   & 3212.1       & 174              & 112 \\
                       & 0.05     & 2257.6       & 0.2\%   & 2263.0              & 3442.7              & 0.0\%   & 3442.7       & 150              & 124 \\ \hline
\multirow{3}{*}{MIX}   & 0.01     & 2724.1       & 0.0\%   & 2724.1              & 3008.4              & 0.0\%   & 3008.4       & 997             & 838 \\
                       & 0.03     & 2429.2       & 0.0\%   & 2429.2              & 3281.9              & 0.0\%   & 3281.9       & 908              & 731 \\
                       & 0.05     & 2134.3       & 0.0\%   & 2134.3              & 3560.7              & 0.0\%   & 3560.7       & 1001              & 851
\end{tabular}
\caption{Results for the transportation network problem}
\label{Tab:network_reslt}
\end{table}

\subsection{The effectiveness of the redundant constraint $x \circ s = 0$}

Finally, we investigate the effectiveness of the redundant
complementarity constraint in (\ref{equ:quadbarp-}) and
(\ref{equ:quadbarp+}) by also solving relaxations without the
the linearized version of the constraint. As it turns out, in all
calculations of $q^-_{\SDP}$, dropping the linearized complementarity
constraint does not change the relaxation value. However, in all
calculations of $q^+_{\SDP}$, dropping it has a significant effect as
shown in Table \ref{Tab:redundant_constraint}. 
In the table, the gap is defined as
\[
\text{gap} = \frac{(\text{value without constraint)} - q_{\SDP}^+}{\max\{|q_{\SDP}^+|, \ 1\}} \times 100\%.
\]  

\begin{table}[tbp]
\centering
\begin{tabular}{lccc}  % repeats {c|} 4 times
Example & $q_{\SDP}^+$ & gap & value without constraint   \\ \hline
Sec.~\ref{ssec:sa_literature} (\#1)     & -16000 & 100\%  & 0      \\
Sec.~\ref{ssec:sa_literature} (\#2)     & -18667 & 100\%  & 0      \\
Sec.~\ref{ssec:sa_literature} (\#3)     & -16000 & 100\%  & 0      \\
Sec.~\ref{ssec:inventory}               & 25600  & 1671\% & 453298 \\
Sec.~\ref{ssec:networkflow} (POLY 0.01) & 3088.8 & 5.7\%  & 3265.8 \\
Sec.~\ref{ssec:networkflow} (POLY 0.03) & 3437.4 & 2.7\%  & 3528.5 \\
Sec.~\ref{ssec:networkflow} (POLY 0.05) & 3769.8 & 2.3\%  & 3855.6 \\
Sec.~\ref{ssec:networkflow} (SOC 0.01)  & 2981.6 & 87.6\% & 5593.1 \\
Sec.~\ref{ssec:networkflow} (SOC 0.03)  & 3212.1 & 84.3\% & 5920.2 \\
Sec.~\ref{ssec:networkflow} (SOC 0.05)  & 3442.7 & 81.6\% & 6252.7 \\
Sec.~\ref{ssec:networkflow} (MIX 0.01)  & 3008.4 & 10.3\% & 3319.4 \\
Sec.~\ref{ssec:networkflow} (MIX 0.03)  & 3281.9 & 5.2\%  & 3453.5 \\
Sec.~\ref{ssec:networkflow} (MIX 0.05)  & 3560.7 & 3.6\%  & 3689.4 \\
\end{tabular}
\caption{Effectiveness of the linearized complementarity constraint}
\label{Tab:redundant_constraint}
\end{table}

\section{Conclusion} \label{sec:conclusion}

In this paper, we have introduced the idea of robust sensitivity
analysis for the optimal value of LP. In particular, we have discussed
the best- and worst-case optimal values under general perturbations in
the objective coefficients and right-hand sides. We have also presented
finite variants that avoid cases of infeasibility and unboundedness.
As the involved problems are nonconvex and very difficult to solve in
general, we have proposed copositive reformulations, which provide a
theoretical basis for constructing tractable SDP-based relaxations that
take into account the nature of the uncertainty set, e.g., through RLT
and SOC-RLT constraints. Numerical experiments have indicated that
our approach works very well on examples from, and inspired by, the
literature. In future research, it would be interesting to improve the
solution speed of the largest relaxations and to explore the possibility
of also handling perturbations in the constraint matrix.

\section*{Acknowledgments} \label{sec:ack}

The first author acknowledges the financial support of the AFRL Mathematical
Modeling and Optimization Institute, where he visited in Summer 2015. The second
author acknowledges the financial support of Karthik Natarajan (Singapore University
of Technology and Design) and Chung Piaw Teo (National University of Singapore),
whom he visited in Fall 2014. This research has benefitted from their many insightful
comments.

\end{onehalfspace}

\bibliographystyle{abbrv}

%\bibliography{rsa}

\appendix

\end{document}